\documentclass{article}

\RequirePackage[OT1]{fontenc}
\RequirePackage[
amsthm,amsmath,natbib
]{imsart}

\usepackage{graphicx}
\usepackage{latexsym,amsmath}
\usepackage{amsmath,amsthm,amscd}
\usepackage{amsfonts}
\usepackage[psamsfonts]{amssymb}
\usepackage{enumerate}
\usepackage{mathscinet}

\usepackage{url}

\usepackage{tocvsec2}

\startlocaldefs

\numberwithin{equation}{section}

\theoremstyle{plain} 
\newtheorem{theorem}{Theorem}[section]
\newtheorem{corollary}[theorem]{Corollary}

\newtheorem{lemma}[theorem]{Lemma}
\newtheorem{subl}[theorem]{Sublemma}

\newtheorem{proposition}[theorem]{Proposition}

\theoremstyle{definition} 

\theoremstyle{definition} 

\newtheorem*{ex*}{Example}
\theoremstyle{remark} 

\theoremstyle{remark} 
\newtheorem{remark}[theorem]{Remark}
\newtheorem*{remark*}{Remark}
\numberwithin{equation}{section}
 
\newcommand{\beqa}{\begin{eqnarray}}
\newcommand{\eeqa}{\end{eqnarray}}
 
\newcommand{\bseq}{\begin{subequations}}
 
\newcommand{\eseq}{\end{subequations}}

\newcommand{\dd}{\partial}

\renewcommand{\dd}{{\,\operatorname{d}}}

\newcommand{\BE}{{\operatorname{BE}}}

\newcommand{\ch}{\operatorname{ch}}
\newcommand{\sh}{\operatorname{sh}}
\newcommand{\sech}{\operatorname{sech}}
\renewcommand{\th}{\operatorname{th}}
\newcommand{\arcch}{\operatorname{arcch}}
\newcommand{\arccosh}{\operatorname{arccosh}}
\renewcommand{\cosh}{\ch}
\renewcommand{\sinh}{\sh}
\renewcommand{\tanh}{\th}

\renewcommand{\a}{k}
\newcommand{\at}{\tilde k}
\newcommand{\h}{b}

\newcommand{\la}{\lambda}
\newcommand{\ga}{\gamma}
\newcommand{\de}{\delta}

\newcommand{\De}{\Delta}
\newcommand{\La}{\Lambda}
\newcommand{\vpi}{\varphi}
\newcommand{\tth}{\theta}

\newcommand{\ii}[1]{\,\mathbf{I}\{#1\}} 

\newcommand{\pd}[2]{\frac{\partial#1}{\partial#2}} 
\newcommand{\fd}[2]{\frac{\dd#1}{\dd#2}}

\renewcommand{\P}{\operatorname{\mathsf{P}}} 
\newcommand{\E}{\operatorname{\mathsf{E}}}
\newcommand{\Var}{\operatorname{\mathsf{Var}}}

\newcommand{\Z}{\mathbb{Z}}
\newcommand{\R}{\mathbb{R}}

\newcommand{\EE}{\mathcal{E}}

\newcommand{\G}{\overline{\Phi}}

\newcommand{\vp}{\varepsilon}

\newcommand{\tX}{{\tilde{X}}}

\newcommand{\tP}{{\tilde{P}}}
\newcommand{\tJ}{{\tilde{J}}}
\newcommand{\tLa}{{\tilde{\La}}}

\newcommand{\tw}{{\tilde{w}}}
\newcommand{\tW}{{\tilde{W}}}

\newcommand{\bF}{{\overline{F}}}

\renewcommand{\le}{\leqslant}
\renewcommand{\ge}{\geqslant}

\endlocaldefs

\begin{document}

\begin{frontmatter}

\title{An asymptotically Gaussian bound on the Rademacher tails}
\runtitle{Gaussian-Rademacher bound}

%

\begin{aug}
\author{\fnms{Iosif} \snm{Pinelis}\thanksref{t2}\ead[label=e1]{ipinelis@mtu.edu}}
  \thankstext{t2}{Supported by NSF grant DMS-0805946}
\runauthor{Iosif Pinelis}


\address{Department of Mathematical Sciences\\
Michigan Technological University\\
Houghton, Michigan 49931, USA\\
E-mail: \printead[ipinelis@mtu.edu]{e1}}
\end{aug}

\begin{abstract}
An explicit upper bound on the tail probabilities for the normalized Rademacher sums is given. This bound, which is best possible in a certain sense, is asymptotically equivalent to the corresponding tail probability of the standard normal distribution, thus affirming a longstanding conjecture by Efron. Applications to sums of general centered uniformly bounded independent random variables and to the Student test are presented. 
\end{abstract}

  
%




\end{frontmatter}

\settocdepth{chapter}

\tableofcontents 

\settocdepth{subsubsection}

\theoremstyle{plain} 
\numberwithin{equation}{section}


\section{Introduction, summary, and discussion}\label{intro} 

Let $\vp_1,\dots,\vp_n$ be independent Rademacher random variables (r.v.'s), so that $\P(\vp_i=1)=\P(\vp_i=-1)=\frac12$ for all $i$. Let $a_1,\dots,a_n$ be any real numbers such that 
\begin{equation}\label{eq:norm}
	a_1^2+\dots+a_n^2=1.
\end{equation}
Let
\begin{equation*}
	S_n:=a_1\vp_1+\dots+a_n\vp_n
\end{equation*}
be the corresponding normalized Rademacher sum. 
Let $Z$ denote a standard normal r.v., with the density function $\vpi$, so that
$\vpi(x)=\frac1{\sqrt{2\pi}}\,e^{-x^2/2}$ for all real $x$.   

Upper bounds on the tail probabilities $\P(S_n\ge x)$ have been of interest in combinatorics/optimization/operations research; see e.g.\ \cite{holt-kleit,ben-tal-et_al,derink06,derink07,ben-tal;nemir09,hiri} and bibliography therein. 
Other authors, including Bennett \cite{bennett}, Hoeffding \cite{hoeff63}, and Efron~\cite{efron69}, were mainly interested in applications in statistics. 
The present paper too was motivated in part by statistical applications in \cite{nonlinear}. 

A particular case of a well-known result by Hoeffding \cite{hoeff63} is the inequality 
\begin{equation}\label{eq:hoeff}
	\P(S_n\ge x)\le e^{-x^2/2}
\end{equation}
for all $x\ge0$. 
Obviously related to this is Khinchin's inequality --- see e.g.\ survey \cite{peskir-shir}; for other developments, including more recent ones, see e.g.\ \cite{latala-olesz_khin-kahane,
kwapien_best-khin-rot-inv,
olesz-nonsymm,veraar10
}. Papers \cite{pin94,dim-reduct} contain multidimensional analogues of an exact version of Khinchin's inequality, whereas \cite{spher} presents their extensions to multi-affine forms in $\vp_1,\dots,\vp_n$ (also known as Rademacher chaoses) with values in a vector space. Lata{\l}a \cite{latala06} gave bounds on moments and tails of {G}aussian chaoses; Berry--Esseen-type bounds for general chaoses were recently obtained by Mossel, O'Donnell, and Oleszkiewicz \cite{mossel-etal}. 
For other kinds of improvements/generalizations of the inequality \eqref{eq:hoeff} see the recent paper \cite{ant-krug} and bibliography there. 

While easy to state and prove, bound \eqref{eq:hoeff} is, as noted by Efron \cite{efron69}, ``not sharp enough to be useful in practice''. 
Exponential inequalities such as \eqref{eq:hoeff} are obtained by finding a suitable upper bound (say $\EE(t)$) on the exponential moments $\E e^{tS_n}$ and then minimizing the Markov bound $e^{-tx}\EE(t)$ on $\P(S_n\ge x)$ in $t\ge0$. 
The best exponential bound of this kind on the standard normal tail probability $\P(Z\ge x)$ is $\inf_{t\ge0}e^{-tx}\E e^{tZ}=e^{-x^2/2}$, for any $x\ge0$. 
%
Thus, a factor of the order of magnitude of $\frac1x$ is ``missing" in this bound, compared with the asymptotics $\P(Z\ge x)\sim\frac1x\,\vpi(x)$ as $x\to\infty$; cf.\ the result by Talagrand \cite{talagr-miss}. 
Now it should be clear that any exponential upper bound on the tail probabilities for sums of independent random variables must be missing the $\frac1x$ factor. 
The problem here is that the class of exponential moment functions is too small. 

Eaton~\cite{eaton1} obtained the moment comparison $\E f(S_n)\le\E f(Z)$ for a much richer class of moment functions $f$, which enabled him \cite{eaton2} to obtain  
an upper bound on $\P(S_n\ge x)$, which is asymptotic to $c_3\P(Z\ge x)$ as $x\to\infty$, where 
$$c_3:=\tfrac{2e^3}9=4.4634\dots.$$
Eaton further conjectured that $\P(S_n\ge x)\le c_3\frac1x\,\vpi(x)$ for $x>\sqrt2$. 
The stronger form of this conjecture, 
\begin{equation}\label{eq:pin94}
	\P(S_n\ge x)\le c\P(Z\ge x)
\end{equation}
for all $x\in\R$ with $c=c_3$
was proved by Pinelis~\cite{pin94}, along with a multidimensional extension, which generalized results of Eaton and Efron \cite{eaton-efron}. 
Various generalizations and improvements of inequality \eqref{eq:pin94} as well as related results were given by Pinelis \cite{pin98,pin99,pin-eaton,binom,normal,asymm,pin-ed,pin-hoeff}
and Bentkus \cite{bent-liet02,bent-jtp,bent-ap}. 

Clearly, as pointed out e.g.\ in \cite{bent-isr}, the constant $c$ in \eqref{eq:pin94} cannot be less than 
\begin{equation}\label{eq:c_*}
	c_*:=\frac{\P\big(\frac1{\sqrt2}(\vp_1+\vp_2)\ge\sqrt2\,\big)}{\P(Z\ge\sqrt2)}=3.1786\dots,
\end{equation}
which may be compared with $c_3$. 
Bobkov, G\"{o}tze and Houdr\'{e} (BGH) \cite{BGH} gave a simple proof of \eqref{eq:pin94} with a constant factor $c\approx12.01$. Their method was based on the Chapman-Kolmogorov identity for the Markov chain $(S_n)$. Such an identity was used, e.g., in \cite{maximal} concerning a conjecture by Graversen and Pe\v skir \cite{g-peskir} on $\max_{k\le n}|S_k|$. 
Pinelis \cite{pin-towards} showed that a modification of the BGH method can be used to obtain inequality \eqref{eq:pin94} with a constant factor $c\approx1.01\,c_*$. 
Bentkus and Dzindzalieta \cite{bent-dzin} recently closed the gap by proving that $c_*$ is indeed the best possible constant factor $c$ in \eqref{eq:pin94}; they used the Chapman-Kolmogorov identity together with the Berry-Esseen bound and a new extension of the Markov inequality. 
Bentkus and Dzindzalieta \cite{bent-dzin} also obtained the inequality    
\begin{equation}\label{eq:1,sqrt2}
	\P(S_n\ge x) \le \tfrac14+\tfrac18\big(1-\sqrt{2-2/x^2}\,\big)\quad\text{for }x\in(1,\sqrt2\,], 
\end{equation}
whereas Holzman and Kleitman \cite{holt-kleit} proved that 
$
	\P(S_n>1) \le\frac5{16}. 
$


We should also like to mention another kind of result, due to Montgomery-Smith~\cite{mont-smith_radem}, who obtained an upper bound on $\ln\P(S_n\ge x)$ and a matching lower bound on $\ln\P(S_n\ge Cx)$ for some absolute constant $C>0$; these bounds depend on $x>0$ and on the sequence $(a_1,\dots,a_n)$ and differ from each other by no more than an absolute constant factor; the constants were improved by Hitczenko and Kwapien \cite{hit-kwap94}. 
The result of \cite{mont-smith_radem} was inspired by 
upper and lower bounds on the $L^p$-norm of sums of general independent zero-mean r.v.'s obtained by Lata{\l}a \cite{latala-moments} and was 
extended to such general sums in \cite{hit-mont}. 
The proof in \cite{mont-smith_radem} was 
in part
based on an extension of the improvement of {H}offmann-{J}\o rgensen's inequality \cite{hoff74} found by Klass and Nowicki \cite{klass-now00}. 
More recent developments in this direction are given in \cite{klass-now07,klass-now10}. 

In the mentioned paper \cite{efron69}, Efron conjectured that there exists an upper bound on the tail probability $\P(S_n\ge x)$ which behaves as the corresponding standard normal tail $\P(Z\ge x)$, and he presented certain facts in favor of this conjecture. Efron's conjecture suggests that even the best possible constant factor $c=c_*=3.17\dots$ in \eqref{eq:pin94} is excessive for large $x$; rather, for such $x$ the ratio of a good bound on $\P(S_n\ge x)$ to $\P(Z\ge x)$ should be close to $1$. Theorem~\ref{th:} below provides such a bound, of simple and explicit form. 

Another well-known conjecture, apparently due to Edelman \cite{portnoy,edelman-comm}, is that 
\begin{equation}\label{eq:conjec}
	\P(S_n\ge x)\le 
	\sup\nolimits_{n\ge1}\P\big(\tfrac1{\sqrt n}\,(\vp_1+\dots+\vp_n)\ge x\big)
\end{equation}
for all $x\ge0$; that is, the conjecture is that the supremum of $\P(S_n\ge x)$ over all finite sequences  $(a_1,\dots,a_n)$ satisfying condition \eqref{eq:norm} is the same as that over all 
such $(a_1,\dots,a_n)$ with equal $a_i$'s. Certain parts of the proof of Theorem~\ref{th:} may be seen as providing additional credence to this conjecture. On the other hand, if \eqref{eq:conjec} were known to be true, it would to a certain extent simplify the proof of Theorem~\ref{th:}. 
Also, it is noted in \cite{bent-dzin} that \eqref{eq:conjec}, used together with the Berry-Esseen bound, would imply another known conjecture \cite{holt-kleit,ben-tal-et_al,hiri} -- that $\P(S_n>1)\le\frac14$. 
Yet another interesting conjecture \cite{burk_stein-property,hit-kwa,olesz_stein-property,veraar08} 
states that $\P(S_n\ge1)\ge\frac7{64}$.


The main result of the present paper is 

\begin{theorem}\label{th:}
For all real $x>0$
\begin{equation}\label{eq:}
	\P(S_n\ge x)\le Q(x):=\P(Z>x)+\frac{C\vpi(x)}{9+x^2}<\P(Z>x)\Big(1+\frac Cx\Big),
\end{equation}
where 
\begin{equation}\label{eq:C}
	C:=5\sqrt{2\pi e}\P(|Z|<1)=14.10\dots. 
\end{equation}
\end{theorem}

\begin{remark}\label{rem:best}
The constant factor $C$ is the best possible in the sense that the first inequality in \eqref{eq:} turns into the equality when $x=n=1$. 
It would be desirable to find the optimal constant $C$ if the constant $9$ in the denominator in \eqref{eq:} is replaced by a smaller positive value, for then the bound $Q(x)$ would be decreasing somewhat faster; however, such a quest appears to entail significant technical complications. 
\end{remark}

Using e.g.\ part (II) of Proposition~\ref{subbl:lhosp}, it is easy to see that the ratio of the bound $Q(x)$ in \eqref{eq:} to $\P(Z>x)$ increases from ${}\approx2.25$ to ${}\approx3.61$ and then decreases to $1$  
as $x$ increases from $0$ to ${}\approx2.46$ to $\infty$, respectively. Figure~\ref{fig:c1} presents a graphical comparison of this ratio, $Q(x)/\P(Z>x)$,  
with 
\begin{enumerate}[(i)]
	\item 
the best possible constant factor $c=c_*\approx3.18$ in \eqref{eq:pin94}; 
\item
the level $1$, which is asymptotic (as $x\to\infty$) to the ratio of either one of the two bounds in \eqref{eq:} to $\P(Z>x)$, and hence, by the central limit theorem, is also asymptotic to 
the ratio of the supremum of $\P(S_n\ge x)$ (over all normalized Rademacher sums $S_n$) to $\P(Z>x)$; 
\item
the ratio of Hoeffding's bound $e^{-x^2/2}$ to $\P(Z>x)$. 
\end{enumerate}
In Figure~\ref{fig:c1}, the graph of the latter ratio looks like a steep straight line (and asymptotically, for large $x$, is a straight line), most of which is outside the vertical range of the picture, thus showing how much the bounds $c_*\P(Z\ge x)$ and $Q(x)$ improve the Hoeffding bound $e^{-x^2/2}$.

\begin{figure}[htbp]
	\centering		
	\includegraphics[scale=1.00]{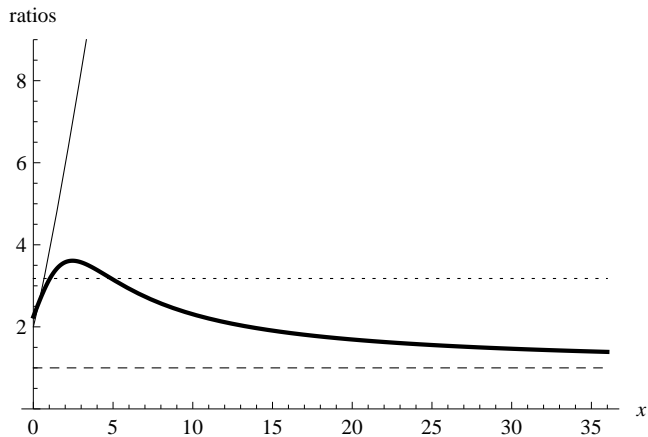}
	\caption{Ratio $Q(x)/\P(Z>x)$ (thick solid) compared with the ratio $e^{-x^2/2}/\P(Z>x)$ (solid, steeply upwards), as well as with the 
	levels $1$ (dashed) and $c_*\approx3.18$ (dotted)}
	\label{fig:c1}
\end{figure} 

In view of the main result of Bentkus \cite{bent-liet01}, one immediately obtains the following corollary of Theorem~\ref{th:}. 

\begin{corollary}\label{cor:bent} 
Let $X, X_1,\dots,X_n $ be i.i.d.\ r.v.'s such that $\P(|X|\le 1)=1$ and $\E X=0$. 
Then 
\begin{equation*}
	\P\Big(\frac{X_1+\dots+X_n}{\sqrt n}\ge x\Big) 
	\le2\hat Q_n(x)
\end{equation*}
for all real $x\ge0$, where $\hat Q_n$ is the linear interpolation of the restriction of the function $Q$ to the set 
$\frac2{\sqrt n}(\frac n2-\lfloor\frac n2\rfloor+\Z)$. 
\end{corollary}



Here we shall present just one more application of Theorem~\ref{th:}, to the self-normalized sums
\begin{equation*}
V_n:=\frac{X_1+\dots+X_n}{\sqrt{X_1^2+\dots+X_n^2}},
\end{equation*}
where, following 
Efron \cite{efron69}, we assume that the $X_i$'s satisfy the so-called orthant symmetry condition:
the joint distribution of $s_1X_1,\dots,s_n X_n$ is the same for any choice of signs $s_1,\dots,s_n\in\{1,-1\}$, so that, in particular, each $X_i$ is symmetrically distributed. It suffices that the $X_i$'s be independent and symmetrically (but not necessarily identically) distributed. 
In particular, $V_n=S_n$ if $X_i=a_i\vp_i$ for all $i$.
It was noted by Efron that (i) Student's statistic $T_n$ is a monotonic function of the so-called self-normalized sum: $T_n=\sqrt{\frac{n-1}n}\,V_n/\sqrt{1-V_n^2/n}$ and (ii)
the orthant symmetry implies in general that the distribution of $V_n$ is a mixture of the distributions of normalized Rademacher sums $S_n$. Thus, one obtains

\begin{corollary}\label{cor:}
Theorem~\ref{th:} holds with $V_n$ in place of $S_n$.
\end{corollary}

Note that many of the most significant advances concerning self-normalized sums are rather recent; e.g., a necessary and sufficient condition for their asymptotic normality was obtained only in 1997 by Gin{\'e},  G{\"o}tze, and Mason \cite{ggm}. 


It appears natural to compare the probability inequalities given in Theorem~\ref{th:} with limit theorems for 
large deviation probabilities. Most of such  theorems, referred to as large deviation principles (LDP's), deal with logarithmic asymptotics, that is, asymptotics of the logarithm of small probabilities; see e.g.\ \cite{dembo-zeit}. As far as the logarithmic asymptotics is concerned, the mentioned bounds $c_*\P(Z\ge x)$ and $Q(x)$ and the Hoeffding bound $e^{-x^2/2}$ are all the same: $\ln\big[c_*\P(Z\ge x)\big]\sim\ln Q(x)\sim\ln e^{-x^2/2}=-x^2/2$ as $x\to\infty$; yet, as we have seen, at least the first two of these bounds are vastly different from the Hoeffding bound, especially from the perspective of statistical practice. Results on the so-called exact asymptotics for 
large deviations (that is, asymptotics for the small probabilities themselves, rather than for their logarithms) are much fewer; see e.g.\ \cite[Theorem~3.7.4]{dembo-zeit} and \cite[Ch.\ VIII]{pet75}. 
Note that the inequalities in \eqref{eq:} hold for all $x>0$, and, \emph{a priori}, the summands $a_i\vp_i$ do not have to be identically or nearly identically distributed; cf.\ conjecture \eqref{eq:conjec}. 
In contrast, almost all limit theorems for large deviations in the literature -- whether with exact or logarithmic asymptotics -- 
hold only for $x=O(\sqrt n)$, with $n$ being the number of identically or quasi-identically distributed (usually independent or nearly independent) random summands; the few exceptions here include results of the papers \cite{a.nagaev,pin81,pin85,pin02,vinogr} and references therein, where the restriction $x=O(\sqrt n)$ is not imposed and $x$ is allowed to be arbitrarily large. 
In general, observe that a limit theorem is a statement on the existence of an inequality, not yet fully specified, as e.g.\ in ``there exists some $n_0$ such that $|x_n-x|<\vp$ for all $n\ge n_0$''; as such, a limit theorem cannot provide a specific bound. Of course, being less specific, limit theorems are applicable to objects of much greater variety and complexity, 
and limit theorems usually provide valuable initial insight. Yet, it seems natural to suppose that the tendency, say in the studies of large deviation probabilities, will be to proceed from logarithmic asymptotics to asymptotics of the probabilities themselves and then on to exact inequalities. We appear to be largely at the beginning of this process, still struggling even with such comparatively simple objects as the Rademacher sums -- the simplicity of which is only comparative, as the discussion around Figure~1 in \cite{pin-towards} suggests. 
However, there have already been a number of big strides made in this direction. For instance, Boucheron, Bousquet, Lugosi, and Massart \cite{bouch-etal} obtained explicit bounds on moments of general functions of independent r.v.'s; their approach was based on a generalization of Ledoux's entropy method \cite{ledoux-esaim,ledoux_book}, using at that a generalized tensorization inequality due to Lata{\l}a and Oleszkiewicz \cite{lat-olesz}. More recently, Tropp \cite{tropp1004} provided noncommutative generalizations of the Bennett, Bernstein, Chernoff, and Hoeffding bounds -- even with explicit and optimal constants; as pointed out in \cite{tropp1004}, ``[a]symptotic
theory is less relevant in practice''.  
Yet, as stated above, in the case of Rademacher sums and other related cases significantly more precise bounds can be obtained.

\section{Proof of Theorem~\ref{th:}: outline}\label{proof} 
Let us begin the proof with several introductory remarks. 

In this section, a number of lemmas will be stated, from which Theorem~\ref{th:} will easily follow. Most of these lemmas will be proved in Section~\ref{proofs of lemmas} -- with the exception of Lemmas~
\ref{lem:B<h} and \ref{lem:BGH-h}, whose proofs are more complicated and will each be presented in a separate section. 
Each of these two more complicated lemmas is based on a number of sublemmas -- which are stated in the corresponding section and used there to prove the lemma; each of these two sections is then completed by proving the sublemmas. 
This tree-like structure appears suitable for presentation: first the general scheme of the proof and then gradually down to the finer details.  

There are many symbols used in the proof. Therefore, let us assume a localization principle for notations: any notations introduced in a section or in a proof of a lemma/sublemma 
supersede those introduced in preceding sections or proofs. For example, the meaning of the $X_i$'s introduced later in this section differs from that in Section~\ref{intro}. 

Without loss of generality (w.l.o.g.), assume that 
\begin{equation}\label{eq:a}
	0\le a_1\le\dots\le a_n=:a,
\end{equation}
so that $a=\max_i a_i$. 
Introduce the numbers 
\begin{equation*}
	u_i:=u_{i,x}:=xa_i,
\end{equation*}
whence for all $x\ge0$
\begin{equation*}
	0\le u_1\le\dots\le u_n=xa. 
\end{equation*} 

The proof of Theorem~\ref{th:} 
is to large extent based on a careful analysis of the Esscher tilt transform of the r.v.\ $S_n$. 
In introducing and using this transform, Esscher and then Cram\'er were motivated by applications in actuarial science. Closely related to the Esscher transform is the saddle-point approximation; for a  recent development in this area, see \cite{exp-defic}. 
The Esscher tilt has been used extensively in limit theorems for large deviation probabilities, but much less commonly concerning explicit probability inequalities -- two rather different in character cases of the latter kind are represented by Rai\v c \cite{raic} and Pinelis and Molzon \cite{nonlinear}. 
One may also note that, in deriving LDP's, the tilt is usually employed to get a lower bound on the probability; in contrast, in this paper the tilt is used to obtain the upper bound. 


The main idea of the proof is to reduce the problem from that on the vector $(a_1,\dots,a_n)$ of an unbounded dimension $n$ to a set of low-dimensional extremal problems, involving sums of the form $\sum_ig(u_i)$. The first step here is to represent such sums as $x^2\int\tilde g\dd\nu$, where $\tilde g(u):=g(u)/u^2$ (for $u\ne0$), 
\begin{equation}\label{eq:nu}
\nu:=\frac1{x^2}\,\sum_iu_i^2\de_{u_i},	
\end{equation}
and $\de_t$ denotes the Dirac probability measure at point $t$, 
so that $\nu$ is a probability measure on the interval $[0,xa]$.   
This step turns the initial finite-dimensional problem into an infinite-dimensional one, involving the measure $\nu$. However, then the well-known Carath\'eodory principle allows one to reduce the dimension to (at most) $k-1$, where $k$ is the total number of the integrals (with the respect to the measure $\nu$) involved in the extremal problem in hand. 
Moreover, it turns out that the systems of integrands one has to deal with in the proof of Theorem~\ref{th:} enjoy the so-called Tchebycheff and, even, Markov properties; therefore, one can reduce the dimension even further, to about $k/2$, which allows for effective analyses. 
It should also be noted that the verification of the Markov property of 
a finite sequence of functions largely reduces to checking the positivity of several functions of only one variable.  
Major expositions of the theory of Tchebycheff-Markov systems and its applications are given in the monographs by Karlin and Studden \cite{karlin-studden} and Kre{\u\i}n and Nudel{\cprime}man \cite{krein-nudelman};  
closely related to this theory are certain results in real algebraic geometry, whereby polynomials are ``certified'' to be positive on a semialgebraic domain by means of an explicit representation, say in terms of sums of squares of polynomials; see e.g.\ \cite{lasserre,marshall-murray}. 
A brief review of the Tchebycheff and Markov systems of functions, which contains all the definitions and facts necessary for the applications in the present paper, is given in \cite{T-syst}. 

Even after the just described reductions in dimension, the proof of Theorem~\ref{th:} entails extensive calculations, both symbolic and numeric, which we did using Mathematica; other advanced calculators should be able to do the job. A well-known result by Tarski \cite{tarski48,loja,collins98} --- which can be viewed as a far-reaching development of Sturm's theorem on the real roots of a polynomial --- 
implies that systems of algebraic equations/inequalities can be solved in a completely algorithmic manner. Similar results hold for algebraic-hyperbolic polynomials (that is, polynomials in $x,e^x,e^{-x}$) --- as well as for certain other expressions involving inverse-trigonometric and inverse-hyperbolic functions (including the logarithmic function), whose derivatives are algebraic. 
However, it was only a few years ago that Tarski's algorithm and its further developments were implemented into widely used computer software. 
In Mathematica, this is done via \texttt{Reduce} and other related commands, such as \texttt{Maximize} and \texttt{Minimize}. 
In particular, command 
$$
\text{
\texttt{Reduce[cond1 \&\& cond2 \&\& $\dots$, \{var1,var2,$\dots$,\}, Reals]}
}
$$
returns a simplified form of the given system (of equations and/or inequalities) \texttt{cond1}, \texttt{cond2}, $\dots$ over real variables \texttt{var1}, \texttt{var2}, $\dots$.
However, the execution of such a command may take a very long time (or require too much computer memory) if the given system is more than a little complicated; in such cases, Mathematica can use some human help. 
As for the commands \texttt{Maximize} and \texttt{Minimize}, whenever possible they return the exact global maximum/minimum subject to the given restrictions; otherwise, these commands return a statement implying that Mathematica cannot do the requested exact optimization.  
Alternatively, such calculations, say for piecewise smooth functions of a finite number of variables, can be done, also quite rigorously, using interval arithmetics; see e.g.\ 
\cite{hansen-walster}; again, the only limitation here is the computer power. 
It should be quite clear that all such calculations done with an aid of a computer are no less reliable or rigorous than similar, or even less involved, calculations done by hand. 
\begin{center}
*****
\end{center}

For all $i=1,\dots,n$, let 
\begin{equation*}
	X_i:=a_i\vp_i.
\end{equation*}
Next, let $\tX_1,\dots,\tX_n$ be any r.v.'s such that 
\begin{equation}\label{eq:tilt}
	\E g(\tX_1,\dots,\tX_n)=\frac{\E e^{xS_n}g(X_1,\dots,X_n)}{\E e^{xS_n}}
\end{equation}
for all Borel-measurable functions $g\colon\R^n\to\R$.  
Equivalently, one may require condition \eqref{eq:tilt} only for Borel-measurable indicator functions $g$; clearly, such r.v.'s $\tX_i$ do exist. It is also clear that the r.v.'s $\tX_i$ are independent. Moreover, for each $i$ the distribution of $\tX_i$ is $\big(e^{u_i}\de_{a_i}+e^{-u_i}\de_{-a_i}\big)/(e^{u_i}+e^{-u_i})$. 

Formula \eqref{eq:tilt} presents the mentioned Esscher tilt transform,  
with the tilting parameter (TP) the same as the $x$ in \eqref{eq:}; 
that is, we choose the TP to be the minimizer of $e^{-tx}\E e^{tZ}=e^{-tx+t^2/2}$ in $t\ge0$ --- rather than the minimizer of $e^{-tx}\E e^{tS_n}$, which latter is usually taken as the TP in limit theorems for large deviations and can thus be expressed only via an implicit function. 
Our choice of the TP appears to simplify the proof greatly. 
 
In terms of the tilted r.v.'s $\tX_1,\dots,\tX_n$, 
introduce now
\begin{gather}
	m_x:=\sum_i\E\tX_i=\frac1x\,\sum_i u_i\th u_i,\quad  
	s_x:=\sqrt{\sum_i\Var\tX_i}=\frac1x\,\sqrt{\sum_i \frac{u_i^2}{\ch^2 u_i}}, \label{eq:m,s}\\
	L_x:=\frac1{s_x^3}\sum_i\E|\tX_i-\E\tX_i|^3, \label{eq:L}
\end{gather}
where $\ch:=\cosh$, $\sh:=\sinh$, $\th:=\tanh$, and $\arcch:=\arccosh$ \big(assuming that $\arcch z\ge0$ for all $z\in[1,\infty)$; thus, for each $z\in[1,\infty)$, $\arcch z$ is the unique solution $y\ge0$ to the equation $\ch y=z$\big). 
Let $\bF_n$ and $\G$ denote, respectively, the tail function of 
$\tX_1+\dots+\tX_n$ and the standard normal tail function, so that
\begin{equation*}
	\bF_n(z)=\P(\tX_1+\dots+\tX_n\ge z)\quad\text{and}\quad
	\G(z)=\P(Z\ge z)
\end{equation*}
for all real $z$. 
Also, let $c_\BE$ denote the least possible constant in the Berry-Esseen inequality 
\begin{equation}\label{eq:BE}
	\sup_{z\in\R}\Big|\bF_n(z)-\G\Big(\frac{z-m_x}{s_x}\Big)\Big|\le c_\BE L_x;
\end{equation}
by Shevtsova \cite{shevtsova-DAN},
\begin{equation*}
	c_\BE\le\tfrac{56}{100}; 
\end{equation*}
a slightly worse bound, $c_\BE\le0.5606$, is due to Tyurin \cite{tyurin}.   

\begin{lemma}\label{lem:<N+B}
For all $x\ge0$  
\begin{equation}\label{eq:<P+B}
	\P(S_n\ge x)\le 
	N(x)+2c_\BE B(x), 
\end{equation}
where 
\begin{align}
	N(x)&:=\exp\Big\{\sum_i\ln\ch u_i+\frac{x^2s_x^2}2-xm_x+\ln\G\Big(\frac{x-m_x}{s_x}+xs_x\Big)\Big\}, \label{eq:N}\\
	B(x)&:=L_x\exp\Big\{-x^2+\sum_i\ln\ch u_i\Big\}. \label{eq:B}
\end{align}
\end{lemma}

Next, introduce the ratio 
\begin{equation}\label{eq:r}
	r(x):=\frac{\vpi(x)}{x\G(x)},
\end{equation}
which is the inverse Mills ratio at $x$ divided by $x$. 
By \cite[Proposition~1.2]{pin-mills}, $r$ is strictly and continuously decreasing from $\infty$ to $1$ on the interval $(0,\infty)$, so that there is a unique root $x_{3/2}\in(0,\infty)$ of the equation 
\begin{equation*}
	r(x_{3/2})=3/2;
\end{equation*}
at that, 
\begin{equation*}
	x_{3/2}=1.03\dots 
\end{equation*}
and
\begin{equation}\label{eq:1<r<3/2}
	1<r(x)\le\tfrac32\quad\text{for}\ x\ge x_{3/2}. 
\end{equation}

Introduce also 
\begin{equation*}
	u_*:=\tfrac{51}{125}=0.408 
\end{equation*}
and 
\begin{equation}\label{eq:h}
	h(x):=\frac{C\vpi(x)}{9+x^2}
\end{equation}
(cf.\ \eqref{eq:}). 
Now one can state an upper bound on the term $N(x)$ in \eqref{eq:<P+B}:

\begin{lemma}\label{lem:N<G}
If $x\ge x_{3/2}$ then $N(x)\le\G(x)$. 
\end{lemma}


\begin{lemma}\label{lem:B<h}
If $x\ge\frac{13}{10}$ and $u_n\le u_*$, then 
$
	2c_\BE B(x)\le h(x). 
$
\end{lemma}

\begin{lemma}\label{lem:h<}
For all $x>0$,   
$
	h(x)<\frac{C\G(x)}x. 
$
\end{lemma}

\begin{lemma}\label{lem:h>}
For all $x>0$,   
$
\P(\vp_1\ge x)\le\G(x)+h(x). 
$
\end{lemma}

Introduce 
\begin{equation*}
	U:=U_{x,a}:=\frac{x-a}{\sqrt{1-a^2}}\quad\text{and}\quad
	V:=V_{x,a}:=\frac{x+a}{\sqrt{1-a^2}},
\end{equation*}
with $a$ as in \eqref{eq:a}. 

\begin{lemma}\label{lem:BGH}
If $x\ge\sqrt3\,$, then 
$
\tfrac12\G(U)+\tfrac12\G(V)\le\G(x). 
$
\end{lemma}

Lemma~\ref{lem:BGH} was proved in \cite{BGH}; cf.\ also \cite[Lemma~5]{pin-towards}. 

\begin{lemma}\label{lem:BGH-h}
If $x\ge\frac{15}{10}$ and $u_n\ge u_*$, then 
$
\tfrac12 h(U)+\tfrac12 h(V)\le h(x). 
$
\end{lemma}

\begin{lemma}\label{lem:x<1.74}
For all $x\in(0,\sqrt3\,]$  
\begin{equation}\label{eq:x<1.3}
\P(S_n\ge x)\le\G(x)+h(x). 
\end{equation}
\end{lemma}

\begin{proof}[Proof of Theorem~\ref{th:}]
By definition \eqref{eq:h} and Lemma~\ref{lem:h<}, it is enough to prove inequality \eqref{eq:x<1.3} for all $x>0$. 
This can be done by induction on $n$. Indeed, for $n=1$ this is Lemma~\ref{lem:h>}. Assume now that $n\ge2$. 
In view of Lemma~\ref{lem:x<1.74}, 
it is enough to prove inequality \eqref{eq:x<1.3} for all $x>\sqrt3\,$. 
At that, in view of Lemmas~\ref{lem:<N+B}, \ref{lem:N<G}, and \ref{lem:B<h}, it is enough to consider the case $u_n>u_*$. To do that,  
write
\begin{equation*}
	\P(S_n\ge x)=\tfrac12\P(\tilde S_{n-1}\ge U)+\tfrac12\P(\tilde S_{n-1}\ge V),
\end{equation*}
where $\tilde S_{n-1}:=b_1\vp_1+\dots+b_{n-1}\vp_{n-1}$, with $b_i:=a_i/\sqrt{1-a^2}$. 
It remains to use the induction hypothesis together with Lemmas~\ref{lem:BGH} and \ref{lem:BGH-h}. 
\end{proof}

\section{Proofs of Lemmas~\ref{lem:<N+B}, \ref{lem:N<G}, \ref{lem:h<}, \ref{lem:h>}, and \ref{lem:x<1.74}}\label{proofs of lemmas} 

\begin{proof}[Proof of Lemma~\ref{lem:<N+B}]
Reading equation \eqref{eq:tilt} with $g(X_1,\dots,X_n)=e^{-xS_n}\ii{S_n\ge x}$ right-to-left, 
recalling \eqref{eq:BE}, and observing that $\E e^{xS_n}=\prod_i\ch u_i$, one has 
\begin{equation*}
	\frac{\P(S_n\ge x)}{\E e^{xS_n}}=-\int_{[x,\infty)}\!\!\!\!\!\! e^{-xy}\dd\bF_n(y)
	=\int_x^\infty\!\!\!\! xe^{-xy}\big(\bF_n(x)-\bF_n(y)\big) \dd y
	\le N_1(x)+B_1(x),	
\end{equation*}
where 
\begin{align*}
N_1(x)&:=\int_x^\infty xe^{-xy}\Big[\G\Big(\frac{x-m_x}{s_x}\Big)-\G\Big(\frac{y-m_x}{s_x}\Big)\Big] \dd y \\
&=\int_x^\infty e^{-xy}\vpi\Big(\frac{y-m_x}{s_x}\Big)\frac{\dd y}{s_x}=\frac{N(x)}{\E e^{xS_n}} 
\end{align*}
and 
\begin{equation*}
B_1(x):=2c_\BE\,L_x\,\int_x^\infty xe^{-xy} \dd y
=2c_\BE\,L_x\,e^{-x^2}=\frac{2c_\BE\,B(x)}{\E e^{xS_n}}. 	
\end{equation*}
Thus, \eqref{eq:<P+B} follows. 
\end{proof}

Now and later in the paper, we need the following special l'Hospital-type rule for monotonicity. 

\begin{proposition}\label{subbl:lhosp}
\emph{(\cite[Propositions~4.1 and 4.3]{pin06})}\quad  
Let $-\infty\le a < b\le\infty$. Let $f$ and $g$ be differentiable functions defined on the interval $(a, b)$. 
It is assumed that $g$ and $g'$ do not take on the zero value
and do not change their respective signs on $(a, b)$. 
\begin{enumerate}[(I)]
	\item If $f(a+)=g(a+)=0$ or $f(b-)=g(b-)=0$, and if the ratio $f'/g'$ is strictly increasing/decreasing on $(a,b)$, then (respectively) $(f/g)'$ is strictly positive/negative and  hence the ratio $f/g$ is strictly increasing/de\-creasing on $(a,b)$.  
	\item If $f(a+)=g(a+)=0$ and if the ratio $f'/g'$ switches its monotonicity pattern at most once on $(a,b)$ --- only from increase to decrease, then the ratio $f/g$ does so. Similar statements, under the condition $f(b-)=g(b-)=0$ and/or for a switch from decrease to increase, are true as well. 
\end{enumerate}
\end{proposition} 

\begin{proof}[Proof of Lemma~\ref{lem:N<G}]
Let us begin this proof by using the well-known fact that the tail function $\G$ is log-concave. This fact is contained e.g.\ in \cite{HKP,pin99}. Alternatively, it can be easily obtained using 
part (I) of Proposition~\ref{subbl:lhosp}, since $(\ln\G)'=-\frac\vpi\G$. So, one can write 
\begin{equation*}
\ln\G(y)\le\ln\G(x)+(\ln\G)'(x)(y-x)=\ln\G(x)-xr(x)(y-x), 	
\end{equation*}
with $y=\frac{x-m_x}{s_x}+xs_x$ (cf.\ \eqref{eq:N}) and $r(x)$ defined by \eqref{eq:r}. Therefore and in view of \eqref{eq:m,s}, 
\begin{equation*}
	\frac1{x^2}\,\ln\frac{N(x)}{\G(x)}\le\tilde\EE(r,\nu)
	:=\int \Big[e(u)+r\cdot\Big(1-\frac{f(u)}{s_x}\Big)\Big]\nu(\dd u),
\end{equation*}
where 
$\int:=\int_0^{xa}$ (recall \eqref{eq:a}), 
\begin{equation*}
		e(u):=\frac{\ln\ch u}{u^2}+\frac1{2\ch^2u}-\frac{\th u}u \quad\text{and}\quad
	f(u):=1-\frac{\th u}u+\frac1{\ch^2u}
\end{equation*}
for $u\ne0$, $e(0):=0$ and $f(0):=1$, and 
$r:=r(x)$. 
Note that the probability measure $\nu$ on the interval $[0,xa]$ defined by \eqref{eq:nu} satisfies the restriction 
\begin{equation}\label{eq:int h}
	\int \h\dd\nu=s_x^2, \quad\text{where}\quad \h(u):=\frac1{\ch^2u}. 
\end{equation}  
Recalling now \eqref{eq:1<r<3/2}, we see that, 
to prove Lemma~\ref{lem:N<G}, we only need to show that $\tilde\EE(r,\nu)\le0$ for all such probability measures $\nu$ and all $r\in[1,\frac32]$; in fact, since $\tilde\EE(r,\nu)$ is affine in $r$, it suffices to consider only $r\in\{1,\frac32\}$. 

Using \cite[Proposition~1]{T-syst} and the Mathematica command \verb9Reduce9, one can check that each of the two systems $(1,-\h,f-e)$ and $(1,-\h,f)$ is an $M_+$-system on $[0,\infty)$; for the system $(1,-\h,f-e)$, it takes about 20 sec on a standard laptop, and about $1$ sec for the system $(1,-\h,f)$. 
Since $s_x\in(0,1]$ and $r\ge1$, the integrand in the integral expression of $\tilde\EE(r,\nu)$ can be rewritten as $g:=r-\frac1\tth\,(f-\tth e)$ with $\tth:=\frac{s_x}r\in(0,1]$, and so,  
$(1,-\h,-g)$ is an $M_+$-system on $[0,\infty)$, for any $r\ge1$ and any value of $s_x$. 
Hence, by \cite[Proposition~2]{T-syst}, the minimum of $\int(-g)\dd\nu$, and thus the maximum of $\tilde\EE(r,\nu)$, over all the probability measures $\nu$ on $[0,xa]$ satisfying the restriction $\int \h\dd\nu=s_x^2$ is attained when the support of $\nu$ is a singleton subset (say $\{u\}$) of $[0,xa]$. 
For this $u$, one has $s_x=1/\ch u$, and it now suffices to show that 
$g(u)=e(u)+r\cdot\big(1-f(u)\ch u \big)\ge0$ for $r\in\{1,\frac32\}$ and $u\in[0,\infty)$;  
using again the Mathematica command \verb9Reduce9, it takes about 2 sec to check this in each of the two cases, $r=1$ and $r=\frac32$. 
\end{proof}

\begin{proof}[Proof of Lemma~\ref{lem:h<}]
Using part (I) of Proposition~\ref{subbl:lhosp}, one can see that the ratio $\frac{xh(x)}{\G(x)}$ is increasing in $x>0$, from $0$ to $C$. Now the result follows. 
\end{proof}
\begin{proof}[Proof of Lemma~\ref{lem:h>}]
Observe that the definition \eqref{eq:C} of $C$ is equivalent to the condition $\G(1)+h(1)=\frac12$ (cf.\ Remark~\ref{rem:best}). Hence and because $\G+h$ is decreasing on $(0,\infty)$, one has $\P(\vp_1\ge x)=\frac12=\G(1)+h(1)\le\G(x)+h(x)$ for all $x\in(0,1]$. For $x>1$, one obviously has $\P(\vp_1\ge x)=0<\G(x)+h(x)$. 
\end{proof}

\begin{proof}[Proof of Lemma~\ref{lem:x<1.74}]
By the symmetry, Chebyshev's inequality, and the main result of \cite{pin-towards}, 
$$\P(S_n\ge x)\le\tfrac12\ii{0<x\le1}+\tfrac1{2x^2}\ii{1<x\le\tfrac{13}{10}}+3.22\G(x)\ii{\tfrac{13}{10}<x\le\sqrt3\,}$$ 
for all $x\in(0,\sqrt3\,].$ 
In particular, for all $x\in(0,1]$ one has $\P(S_n\ge x)\le\tfrac12=\P(\vp_1\ge x)\le\G(x)+h(x)$, by Lemma~\ref{lem:h>}.

Next, let us prove \eqref{eq:x<1.3} for $x\in(1,\frac{13}{10}]$. 
As in the proof of \cite[Lemma~3]{pin-towards}, one can see that the minimum of $x^2\G(x)$ over $x\in[1,\frac{13}{10}]$ is attained at one of the end points of the interval $[1,\frac{13}{10}]$; in fact, the minimum is at $x=1$. It is also easy to see that the minimum of $x^2h(x)$ over $x\in[1,\frac{13}{10}]$ is attained at $x=1$ as well. Thus, $\P(S_n\ge x)\le\frac1{2x^2}=\frac1{2x^2\big(\G(x)+h(x)\big)}\big(\G(x)+h(x)\big)
\le\frac1{2\big(\G(1)+h(1)\big)}\big(\G(x)+h(x)\big)=\G(x)+h(x)$ for $x\in(1,\frac{13}{10}]$. 

The case $x\in(\tfrac{13}{10},\sqrt3\,]$ is similar to the just considered case $x\in(1,\frac{13}{10}]$. Here, using part (II) of Proposition~\ref{subbl:lhosp} 
(cf.\ \cite{pin-towards}), one can see that $h/\G$ switches, just once, from increase to decrease on $(0,\infty)$; in particular, $h/\G$ increases on $(\frac{13}{10},\sqrt3\,]$. So, to complete the proof of Lemma~\ref{lem:x<1.74}, it is enough to check that $3.22\G(\frac{13}{10})\le\G(\frac{13}{10})+h(\frac{13}{10})$, which is true. 
\end{proof}

\section{Proof of Lemma~\ref{lem:B<h}}\label{B<h} 
We begin with a technical sublemma, used in the proof of Sublemma~\ref{subl:jensen-like}: 
\begin{subl}\label{subl:h}
For each $a\in[0,1]$, the function 
\begin{equation}\label{eq:ha}
(0,1]\ni v\mapsto h_a(v):=\arcch(\tfrac1{\sqrt v})(2-v)(v-\tfrac a{\sqrt v})	
\end{equation}
is concave.  
\end{subl}

We shall need the following
tight upper bound on the Lyapunov ratio $L_x$, defined by \eqref{eq:L}: 

\begin{subl}\label{subl:jensen-like}
One has 
\begin{equation}\label{eq:jensen-like}
	L_x\le\frac1{x^3}\,\sum_i u_i^3(1+\th^2u_i)\ch u_i.
\end{equation} 
\end{subl}

By Sublemma~\ref{subl:jensen-like} and the definition \eqref{eq:B} of $B(x)$,  
\begin{equation}\label{eq:B<...J}
	B(x)\le\tfrac1x e^{-x^2+\tilde J},
\end{equation}
where 
\begin{equation*}
	\tilde J:=\tilde J(x,\nu):=x^2\int\ell\dd\nu
	+\ln\int \a\dd\nu, 
\end{equation*}
\begin{equation*}
	\a(u):=u(1+\th^2u)\ch u,\quad 
	\ell(u):=\tfrac{\ln\ch u}{u^2}\text{ for }u\ne0, 
\end{equation*}
$\ell(0):=\frac12$, 
and $\nu$ is the probability measure on the interval $[0,u_*]$ defined by \eqref{eq:nu}, so that $\nu$  
satisfies the restriction \eqref{eq:int h}. 
Noting that $(\ln\ch)''=\th'=\sech^2$ and 
applying (twice) the special lHospital-type rule for monotonicity given by part (I) of  Proposition~\ref{subbl:lhosp}, one sees that   
\begin{equation}\label{eq:ell'<0}
\ell'<0\quad\text{on}\quad(0,\infty).    	
\end{equation} 

To obtain the upper bound $h(x)$ on $2c_\BE B(x)$ as stated in Lemma~\ref{lem:B<h}, we shall maximize $\tilde J(x,\nu)$ over all such probability measures $\nu$. 
To do so, we shall maximize $\int\a\dd\nu$ given values of the integrals $\int1\dd\nu(=1)$, $\int\h\dd\nu(=s_x^2, \text{ as in \eqref{eq:int h}})$, and $\int\ell\dd\nu$. 

\begin{subl}\label{subl:wronskians}
The sequence $(g_0,g_1,g_2,g_3):=(1,-\h,-\ell,\a)$ is an $M_+$-system on $[0,u_*]$. 
\end{subl}

So, by \cite[Proposition~2]{T-syst} (with $n=2$ and $m=1$ there), it suffices to consider $\nu$'s of the form  $\nu=(1-t)\de_u+t\de_{u_*}$ for some $t\in[0,1]$ and $u\in[0,u_*]$. 
For such $\nu$, 
\begin{equation*}
	\tilde J(x,\nu)=J(t,u):=J_x(t,u):=
	x^2\cdot\big((1 - t) \ell(u) + t \ell(u_*)\big) + \ln\big((1-t)\a(u)+t\a(u_*)\big).  
\end{equation*}
Thus, we need to maximize $J(t,u)$ over all $(t,u)\in[0,1]\times[0,u_*]$;  
clearly, this maximum is attained. 
For all $(t,u)\in(0,1)\times[0,u_*)$,  
\begin{align}
	\Big(\pd Jt\Big)\,\frac{(1-t)\a(u)+t\a(u_*)}{u_*-u}
	&=\frac{\big(\a(u_*)-\a(u)\big)+\tau\,\big(\ell(u_*)-\ell(u)\big)}{u_*-u} \notag \\
	&=\a'(w)+\tau\ell'(w), \label{eq:DJt} \\
	\Big(\pd Ju\Big)\,\frac{(1-t)\a(u)+t\a(u_*)}{1-t}
	&=\a'(u)+\tau\ell'(u),  \label{eq:DJu} 
\end{align}
where $\tau:=x^2\cdot\big((1-t)\a(u)+t\a(u_*)\big)$ and $w$ is some number such that $u<w<u_*$ (whose existence follows by the mean-value theeorem). 
So, if the maximum of $J$ over the set $[0,1]\times[0,u_*]$ is attained at some point $(t,u)\in(0,1)\times(0,u_*)$, then at this point one has $\pd Jt=0=\pd Ju$, whence, by  
\eqref{eq:DJt}, \eqref{eq:DJu}, and \eqref{eq:ell'<0},   
$\frac{\a'(w)}{\ell'(w)}=-\tau=\frac{\a'(u)}{\ell'(u)}$ while $u_*>w>u\ge0$, which contradicts 

\begin{subl}\label{subl:rho}
The function $\rho:=\frac{\a'}{\ell'}$ is strictly increasing on the interval $[0,u_*]$ (by continuity, we let $\rho(0):=\rho(0+)=-\infty$). 
\end{subl}

\noindent Also, no maximum of $J$ is attained at any point $(t,u)\in(0,1)\times\{0\}$, because at any such point the right-hand side of \eqref{eq:DJu} is $\a'(0)+\tau\ell'(0)=1+\tau\cdot0>0$, whereas the left-hand side of \eqref{eq:DJu} must be $\le0$. 
Thus, the maximum can be attained at some point $(t,u)\in[0,1]\times[0,u_*]$ only if either $t\in\{0,1\}$ or $u=u_*$. Therefore the maximizing measure $\nu$ must be concentrated at one point, say $u$, of the interval $[0,u_*]$. 
Together with \eqref{eq:B<...J}, this shows that 
\begin{equation*}
	B(x)\le\sup_{u\in[0,u_*]}\tfrac1x e^{-x^2+J_0},
\end{equation*}
where 
\begin{equation*}
	J_0:=J_0(x,u):=J_x(0,u)=x^2\cdot\ell(u)+\ln \a(u). 
\end{equation*}
So, Lemma~\ref{lem:B<h} reduces now to the following statement: 
\begin{equation}\label{eq:La<}
	\La(x,u):=J_0(x,u)-\frac{x^2}2-\ln x+\ln(9+x^2)-K\overset{(\text{?})}\le0
\end{equation}
for all $(x,u)\in[\frac{13}{10},\infty)\times[0,u_*]$, where 
\begin{equation*}
	K:=\ln\frac C{2\sqrt{2\pi}\,c_\BE}. 
\end{equation*}
Thus, one may want to maximize $\La$ in $u\in[0,u_*]$. Towards that end, observe that for all $u>0$ 
\begin{equation*}
	\frac1{-\ell'(u)}\,\pd\La u=\ga(u)-x^2,
\end{equation*}
where 
\begin{equation*}
	\ga:=-\frac{\a'}{\a\ell'}=-\rho\,\frac1\a;
\end{equation*}
so, 
the partial derivative of $\La$ in $u>0$ equals $\ga(u)-x^2$ in sign. 
On the other hand, the function $\frac1\a$ is positive and strictly decreasing and, in view of Sublemma~\ref{subl:rho}, the function $(-\rho)$ is so as well (on the interval $[0,u_*]$). 
It follows that the function $\ga$ too is positive and strictly decreasing on $(0,u_*]$; at that, $\ga(0+)=\infty$. 


Introduce now 
\begin{equation}\label{eq:x_*}
	x_*:=\sqrt{\ga(u_*)}=7.39\dots. 
\end{equation}
By 
the mentioned properties of the function $\ga$, for each $x\in(0,x_*]$ one has $\ga(u)\ge x^2$ for all $u\in[0,u_*]$ and hence $\La(x,u)$ increases in $u\in[0,u_*]$, so that $\La(x,u)\le\La(x,u_*)$ for all $u\in[0,u_*]$.  
Since the derivative of $\La(x,u_*)$ in $x$ is a rather simple rational function, it is easy to see 
that $\La(x,u_*)\le0$ for all $x\ge\frac{13}{10}$. 
So, inequality \eqref{eq:La<} holds for all $(x,u)\in[\frac{13}{10},x_*]\times[0,u_*]$. 

It remains to prove \eqref{eq:La<} for each $x\in[x_*,\infty)$ (and all 
$u\in[0,u_*]$). For each such $x$, 
there is a unique $u_x\in[0,u_*]$ such that $\ga(u)-x^2$ and hence $\pd\La u$ are opposite to $u-u_x$ in sign, and so, $\La(x,u)\le\La(x,u_x)$ for all $u\in[0,u_*]$.  

Since, by \eqref{eq:ell'<0}, the function $\ell$ is strictly and continuously decreasing on $[0,\infty)$, there is a unique inverse function $\ell^{-1}\colon(0,\frac12]\mapsto[0,\infty)$. Now introduce 
\begin{equation*}
	\tJ_0(x,\la):=J_0\big(x,\ell^{-1}(\la)\big)=x^2\la+\ln\at(\la),\quad\text{where}\quad
	\at:=\a\circ\ell^{-1} 
\end{equation*}
and $\la\in[\ell(u_*),\ell(0)]=[\ell(u_*),\frac12]$. 
Next, observe that $(\ln\at)'=-\ga\circ\ell^{-1}$, which is decreasing on $[\ell(u_*),\frac12]$, because $\ga$ and $\ell$ (and hence $\ell^{-1}$) are decreasing. It follows that the function  $\ln\at$ is concave on $[\ell(u_*),\frac12]$, and so, $\tJ_0(x,\la)$ is concave in $\la\in[\ell(u_*),\frac12]$ -- for each real $x$. 
At this point, we need 

\begin{subl}\label{subl:ga>}
For all $u\in(0,u_*]$,  
\begin{equation*}
	\ga(u)>\tfrac6{u^2}. 
\end{equation*}
\end{subl}

By \eqref{eq:x_*} and Sublemma~\ref{subl:ga>}, if $u=\frac{\sqrt6}x$ and $x\ge x_*$, then $u\in(0,u_*]$ and 
$\ga(\frac{\sqrt6}x)>x^2=\ga(u_x)$, which in turn implies that \rule[-7pt]{0pt}{10pt}$\frac{\sqrt6}x<u_x$, $\ell(\frac{\sqrt6}x)>\ell(u_x)$, and $(\ln\at)'\big(\ell(\frac{\sqrt6}x)\big)<(\ln\at)'\big(\ell(u_x)\big)=-\ga(u_x)=-x^2$
(since $\ga$, $\ell$, and $(\ln\at)'$ are decreasing); 
so, for all $\la\in[\ell(u_*),\frac12]$, 
$\pd{\tJ_0}\la\big(x,\ell(\frac{\sqrt6}x)\big)
<\pd{\tJ_0}\la\big(x,\ell(u_x)\big)=0$; 
therefore and by the concavity of $\tJ_0(x,\la)$ in $\la$,  
\begin{equation*}
	\tJ_0(x,\la)\le\tJ_0\big(x,\ell(\tfrac{\sqrt6}x\,)\big)
	+{\textstyle\pd{\tJ_0}\la}\big(x,\ell(\tfrac{\sqrt6}x\,)\big)\,\big(\la-\ell(\tfrac{\sqrt6}x\,)\big)
	\le\hat J_0(x,\tfrac{\sqrt6}x\,) 
\end{equation*}
for all $\la\in[\ell(u_*),\frac12]$, 
where
\begin{equation*}
	\hat J_0(x,u):=J_0(x,u)+\big(x^2-\ga(u)\big)\,\big(\ell(u_*)-\ell(u)\big). 
\end{equation*}
Thus, in view of \eqref{eq:La<}, Lemma~\ref{lem:B<h} reduces to the inequality 
$\hat J_0(x,\tfrac{\sqrt6}x\,)-\frac{x^2}2-\ln x+\ln(9+x^2)-K\le0$ for all $x\ge x_*$, where we change the variable once again, from $x$ to $u$, by the formula $x=\tfrac{\sqrt6}u$. 
So, Lemma~\ref{lem:B<h} reduces to 

\begin{subl}\label{subl:tLa}
For all $u\in(0,u_*]$  
\begin{equation*}
	\tLa(u):=\hat J_0(\tfrac{\sqrt6}u,u\,)-\tfrac3{u^2}-\ln\tfrac{\sqrt6}u+\ln(9+\tfrac6{u^2})\le K. 
\end{equation*}
\end{subl}

It remains, in this section, to prove Sublemmas~\ref{subl:h}--\ref{subl:tLa}. 

\begin{proof}[Proof of Sublemma~\ref{subl:h}]
Since $h_a(v)$ is affine in $a$, 
w.l.o.g.\ $a\in\{0,1\}$.  
Consider first the case $a=0$. 
Observe that 
$\fd{\left(h_0''(t^{-2})\right)}t\cdot4(t^2-1)^{5/2}=(3-2t^2)(4t^4-1)$, 
which switches the sign from $+$ to $-$ at $t=\sqrt{\frac32}$ as $t$ increases from $1$ to $\infty$. 
Hence, the maximum of $h_0''(v)$ in $v\in(0,1]$ is attained at $v=\frac23$, and this maximum is easily seen to be negative, which proves the case $a=0$. 

The case $a=1$ is considered similarly. 
Observe that 
\begin{equation*}
\Big[\fd{}t\Big(\frac{4h_1''(t^{-2})}{8 + t^3 + 6 t^5}\Big)\Big]
(t-1)^{3/2} (1 + t)^{5/2} (8 + t^3 + 6 t^5)^2 
\end{equation*}
is a certain polynomial in $t$ (of degree 13), 
which switches the sign from $+$ to $-$ at a certain algebraic number $t_1$ as $t$ increases from $1$ to $\infty$. 
Hence, the maximum of $\frac{h_1''(t^{-2})}{8 + t^3 + 6 t^5}$ in $t\ge1$ is attained at $t=t_1$, and $h_1''(t_1^{-2})$ can seen to be negative (using e.g.\ the \verb9Reduce9 command again), which proves the case $a=1$ as well. 
Sublemma~\ref{subl:h} is now completely proved. 
\end{proof}

\begin{proof}[Proof of Sublemma~\ref{subl:jensen-like}]
Observe that 
$
	L_x=(xs_x)^{-3}\,\sum_i u_i^3(1-\th^4u_i).
$ 
So, inequality \eqref{eq:jensen-like} means exactly that 
\begin{equation}\label{eq:De}
	\sum_i u_i^3(1-\th^4u_i)-s_x^3\sum_i u_i^3(1+\th^2u_i)\ch u_i
	=\sum_i u_i^2 g(u_i)\le0
\end{equation}
for all $u_i$'s in the interval $[0,u_*]$ such that $\sum_iu_i^2=x^2$ and $\sum_i\frac{u_i^2}{\ch^2u_i}=x^2s_x^2$, where 
\begin{equation*}
	g(u):=u(1-\th^4u)-s_x^3u(1+\th^2u)\ch u
	=u\Big(2-\frac1{\ch^2u}\Big)\Big(\frac1{\ch^3u}-s_x^3\Big)\ch u. 
\end{equation*}
Next, the object $\sum_i u_i^2 g(u_i)$ in \eqref{eq:De} with the restrictions $\sum_iu_i^2=x^2$ and $\sum_i\frac{u_i^2}{\ch^2u_i}=x^2s_x^2$ can be rewritten 
as 
$x^2\E h(Y)$ given $\E Y=s_x^2$, where $h(\cdot):=h_a(\cdot)$ (as in \eqref{eq:ha}) with $a=s_x^3$
and 
$Y$ is a r.v.\ with the distribution   
$\nu:=\frac1{x^2}\sum_i u_i^2 \de_{v_i}$, with $v_i:=\frac1{\ch^2u_i}$; note that one always has $s_x\in(0,1]$ and $\nu$ is indeed a probability measure due to the restriction $\sum_iu_i^2=x^2$. 
So, by 
Sublemma~\ref{subl:h} and Jensen's inequality, 
$x^{-2}\sum_i u_i^2 g(u_i)=\E h(Y)\le h(\E Y)=h(s_x^2)=0$, which proves the inequality in \eqref{eq:De} and hence that in \eqref{eq:jensen-like}.
\end{proof}

\begin{proof}[Proof of Sublemma~\ref{subl:wronskians}]
By \cite[Proposition~1]{T-syst}, it is enough to show that the Wronskian $W_0^j$ is strictly positive on the interval $(0,u_*)$ for each $j=0,\dots,3$. 
It is obvious that $W_0^0(u)=1>0$ and $W_0^1(u)=2\th u/\ch^2u>0$ for all $u>0$. 
Next, using the command \verb9Reduce9, one obtains (in about 3 sec) that 
$\frac14\,W_0^2(u)\,u^4\ch^5u =2 u\ch u \sh^2u - u^2 \sh^3u - 
 (u + 3 \ch u \sh u - 2 u \sh^2u)\,\ch u\, \ln\ch u >0$ for all $u>0$. 
It remains to prove the positivity (for $u\in(0,u_*)$) of 
\begin{gather}
\tW_0^3(u):=16u^5\,\ch^9 u\,W_0^3(u)	\label{eq:tW}
=P_0(u)+P_1(u)\ln\ch u, \quad\text{where} \\
\begin{aligned}
P_0(u)&:=u \Big(167 u-5 u^3-4 (97 u^2+63) \sinh (2 u)+2 (41 u^2-72) \sinh
   (4 u) \\
   & +12 (5 u^2+3) \sinh (6 u)-(13 u^2-36) \sinh (8 u)+2
   (9 u^2+17) u \cosh (2 u) \\
   &-24 (u^2+4) u \cosh (4 u)+2 (7
   u^2+31) u \cosh (6 u)-(3 u^2-25) u \cosh (8 u)\Big), \\
P_1(u)&:=2 \cosh (u) \Big(-250 u^3 \sinh (u)-98 u^3 \sinh (3 u)-34 u^3 \sinh (5 u)+6 u^3
   \sinh (7 u) \\
   &-161 u^2 \cosh (5 u)+23 u^2 \cosh (7 u)+(321 u^2+432) \cosh
   (u)+(96-471 u^2) \cosh (3 u) \\
   &+15 u \sinh (u)-147 u \sinh (3 u)-195 u
   \sinh (5 u)-33 u \sinh (7 u)-96 \cosh (5 u)-48 \cosh (7 u)\Big).  
\end{aligned} \notag
\end{gather} 
Next, one can bracket the function $\ln\ch$ between two of its Pad\'e approximants near $0$ (for a definition, see e.g.\ \cite{pade-baker}): 
\begin{gather*}
\tfrac{3 u^2}{6+u^2}=:r_{2,3}(u)<\ln\ch u
<r_{4,2}(u):=\tfrac{3 u^2 (10 + u^2)}{4(15 + 4 u^2)} 
\end{gather*}
for all $u\in(0,\infty)$; these inequalities can be verified by hand or (in about 4 sec) using the command \verb9Reduce9. 
So, by \eqref{eq:tW}, 
it is enough to show that 
\begin{align*}
	W_{2,3}(u)&:=\tfrac{6+u^2}u \big(P_0(u)+P_1(u)r_{2,3}(u)\big)\quad\text{and}\\ 
	W_{4,2}(u)&:=\tfrac{4(15 + 4 u^2)}u \big(P_0(u)+P_1(u)r_{4,2}(u)\big)
\end{align*}
are positive for all $u\in(0,u_*)$. 
Note that $W_{2,3}(u)$ and $W_{4,2}(u)$ are each an algebraic-hyperbolic polynomial, of the form 
\begin{equation}\label{eq:poly}
\sum_{j=0}^4\big(p_j(u^2)u\ch(2ju)+q_j(u^2)\sh(2ju)\big), 	
\end{equation}
where $p_j(u^2)$ and $q_j(u^2)$ are polynomials in $u^2$ (of degree $2$, except that the $q_j$'s for $W_{4,2}$ are of degree $3$). 
Thus, the positivity of $W_{2,3}$ and $W_{4,2}$ can be checked using the command \verb9Reduce9. It takes about an hour this way to check that $W_{4,2}>0$ on $(0,\infty)$ and about $40$ min to check that $W_{2,3}>0$ on $(0,u_{**})$, where $u_{**}=4.98\ldots>u_*$, so that indeed $W_{2,3}>0$  and $W_{4,2}>0$ on $(0,u_*)$. 

To avoid these long execution times, one can, alternatively, do as follows. 
Let, for a moment, $W$ stand for either $W_{2,3}$ or $W_{4,2}$. 
As mentioned above, $W$ can be represented in the form \eqref{eq:poly}. 
Now expand $\ch(2ju)$ and $\sh(2ju)$ into Maclaurin series, and write  
\begin{align*}
	W(u)&=\sum_{m=0}^\infty w_m(u), \quad\text{where}\\ 	
	w_m(u)&:=\sum_{j=0}^4\Big( p_j(u^2)u\frac{(2ju)^{2m}}{(2m)!}
	+q_j(u^2)\ii{m\ge1}\frac{(2ju)^{2m-1}}{(2m-1)!}\Big), 
\end{align*}
with $(2ju)^{2m}:=1$ if $j=m=0$. 
Note that for $m\ge1$ 
\begin{equation*}
\tw_m(u):=w_m(u)\,\frac{(2m)!}{(2u)^{2m-1}}=\sum_{k=0}^3 c_{m,k} u^{2k}, 	
\end{equation*}
where each coefficient $c_{m,k}$ is of the form $\sum_{j=1}^4(a_k+b_k m)j^{2m}$ for some real numbers $a_k$ and $b_k$, depending only on $k\in\{0,1,2,3\}$. 

For $W=W_{2,3}$, using the command \verb9Reduce9 (or manually), one can quickly verify that for all $m\ge4$ inequalities 
$c_{m,1}<0$, $c_{m,2}>0$, and $c_{m,3}<0$ hold, which implies that for all $u\in(0,u_*)$ one has  
$\tw_m(u)>c_{m,0}+c_{m,1}u_*^2+c_{m,3}u_*^6>0$, the latter inequality quickly checked (say) by another \verb9Reduce9. 
Thus, for all $u\in(0,u_*)$ one has $w_m(u)>0$ for all $m\ge4$ and hence $\sum_{m=6}^\infty w_m(u)>0$. 
On the other hand, $\sum_{m=0}^5 w_m(u)$ is a polynomial in $u$, which can be quickly checked by yet another \verb9Reduce9 to be positive for all $u\in(0,u_*)$. 
This proves that $W_{2,3}>0$ on $(0,u_*)$. 

The case $W=W_{4,2}$ is considered similarly. Here one can check that 
$c_{m,1}<0$ and $c_{m,3}>0$ for $m\ge5$ (whereas 
$c_{m,2}>0$ changes in sign from $+$ to $-$ near $m=24.1$). 
So, $\tw_m(u)>\min(c_{m,0}+c_{m,1}u_*^2+c_{m,2}u_*^4, 
c_{m,0}+c_{m,1}u_*^2+c_{m,2}0^4)>0$ for all $m\ge5$ and $u\in(0,u_*)$, by two more applications of \verb9Reduce9 (say). 
On the other hand, in this case the polynomial $\sum_{m=0}^6 w_m(u)$ \big(rather than $\sum_{m=0}^5 w_m(u)$\big) is positive for all $u\in(0,u_*)$. 
This proves that $W_{4,2}>0$ on $(0,u_*)$. 
\end{proof}

\begin{proof}[Proof of Sublemma~\ref{subl:rho}] 
The main part of this proof, concerning the negativity of the function $DD\rho_1$ defined below, can be done similarly to the alternative proof of the positivity of $W_{2,3}$ in $W_{4,2}$ in Sublemma~\ref{subl:wronskians}. 
Introduce 
\begin{gather*}
	\a_1(u):=\a'(u)u^3,\quad \ell_1(u):=\ell'(u)u^3,\quad\text{so that}\quad
	\rho=\frac{\a_1}{\ell_1}; \\
	\text{let also }\rho_1:=\frac{\a_1'}{\ell_1'}.  
\end{gather*} 
Further, for $u>0$ introduce 
\begin{align*}
	(D\rho_1)(u)&:=\rho_1'(u)\,\frac{32\ch^2u}u\,(u-\tfrac12\sh2u)^2\quad\text{and}\\ 
	(DD\rho_1)(u)&:=\Big(\frac{(D\rho_1)(u)}{u^3}\Big)'\,\frac{u^4}{2\ch u} \\ 
	&=\sum_{j=0}^3\big(p_j(u^2)u\ch(2ju)+q_j(u^2)\sh(2ju)\big)
	=\sum_{m=0}^\infty r_m(u), 
\end{align*}	 
where
\begin{equation*}
	r_m(u):=\sum_{j=0}^3\Big( p_j(u^2)u\frac{\big(2ju\big)^{2m}}{(2m)!}
	+q_j(u^2)\frac{\big(2ju\big)^{2m+1}}{(2m+1)!}\Big) 
	=\frac{(2u)^{2m}u}{(2m+1)!}\sum_{k=0}^2 c_{m,k} u^{2k}, 	
\end{equation*} 
$p_j(u^2)$ and $q_j(u^2)$ are polynomials in $u^2$ (of degree $2$), and each coefficient $c_{m,k}$ is of the form $\sum_{j=1}^3(a_k+b_k m)j^{2m}$ for some real numbers $a_k$ and $b_k$, depending only on $k\in\{0,1,2\}$. 
Using the command \verb9Reduce9, one can verify that for all $m\ge9$ and $k\in\{0,1,2\}$ 
inequalities $c_{m,k}<0$
hold, which implies that for all $u>0$ one has  
$r_m(u)<0$ and hence $\sum_{m=9}^\infty r_m(u)>0$; the verification of the inequality $c_{m,2}<0$ for $m\ge9$ takes about $18$ sec, while that of each of the inequalities $c_{m,0}<0$ and $c_{m,1}<0$ for $m\ge9$, just a fraction of a second. 
On the other hand, $\sum_{m=0}^8 r_m(u)$ is a polynomial in $u$, which can be quickly checked by yet another \verb9Reduce9 to be negative for all $u\in(0,u_*)$. 
This proves that $DD\rho_1<0$ on $(0,u_*)$. It follows that $\frac{(D\rho_1)(u)}{u^3}$ is decreasing on $(0,u_*)$, from $\lim_{u\downarrow0}\frac{(D\rho_1)(u)}{u^3}=64>0$ to $\frac{(D\rho_1)(u_*)}{u_*^3}<0$. 
So, the monotonicity pattern of $\rho_1$ changes exactly once on $(0,u_*)$, from increase to decrease. 
Note also that $k_1(0)=\ell_1(0)=0$. 
So, by part (II) of Proposition~\ref{subbl:lhosp}, the monotonicity pattern of $\rho$ can change at most once on $(0,u_*)$, and only from increase to decrease. However, $\rho'(u_*)=0.017\ldots>0$. 
Thus, $\rho$ is strictly increasing on $(0,u_*)$ and hence, by the continuity, on $[0,u_*]$.   
\end{proof}

\begin{proof}[Proof of Sublemma~\ref{subl:ga>}]
Observe that 
\begin{equation*}
	\ga(u)=\frac{u^2 \sech(2u) \big[\ch u + \ch(3u) + u \big(3\sh u + \sh(3u)\big)\big]}
{4 \ch u\, \ln\ch u - 2 u \sh u}. 
\end{equation*}
The denominator of this ratio is positive, since, as
previously noted, $\ga>0$.
Also, because the function $\ch$ is convex, one has 
$\ch u + \ch(3u) + u \big(3\sh u + \sh(3u)\big)\ge\ch u + \ch(3u) \ge 2\ch(2u)$. 
It follows that  
\begin{equation*}
	\ga(u)\ge\tilde\ga(u):=
	\frac{u^2}{2\ch u\, \ln\ch u - u \sh u}
\end{equation*}
for all $\in(0,u_*]$. 
It remains to use the \verb9Reduce9 command to see that 
$\tilde\ga(u)>\tfrac6{u^2}$ for all $\in(0,u_*]$. 
\end{proof}

\begin{proof}[Proof of Sublemma~\ref{subl:tLa}]
First, observe that 
\begin{equation*}
	\tLa=T_1 + T_2 + T_3 T_4, 
\end{equation*}
where
\begin{alignat*}{2}
	T_1(u) :=& 3\frac{2\ln\ch u - u^2}{u^4}, \qquad &
	T_2(u) :=& \ln\big(\sqrt{\tfrac32}\, (2 + 3 u^2) \ch(2u) \sech u\big), \\
	T_3(u) :=& \ell(u_*) - \ell(u), &
	T_4(u) :=& \tfrac6{u^2} - \ga(u). 
\end{alignat*}
Next, 
\begin{equation*}
	T_4=\frac{f_4}{g_4}, 
\end{equation*}
where 
\begin{align*}
f_4(u):=&-\frac{F_4(u)}{2u^6}, \quad F_4:=F_{41}+F_{42}, \\
F_{41}(u):=&-24 \ln\ch u + 6 u \sech u \,\sech(2u) (\sh(3u) - \sh u) \\
 &+ u^4  \sech(2u) \big(1 + \sech u\,  \ch(3u)\big), \\
F_{42}(u):=& u^5 \sech u\, \sech(2u) \big(3 \sh u + \sh(3u)\big), \\
g_4(u):=&\frac{2 \ln\ch u-u \th u}{u^4}. 
\end{align*}
Further, the derivatives $F_{41}^{(j)}(0)$ are $0$ for all $j=0,\dots,5$, whereas 
\begin{align*}
	F_{41}^{(6)}(u)=24 \sech^7u \big[&80\ch u - 50\ch(3u) + 2\ch(5u) \\
	&- u\,\big(302 \sh u - 57 \sh(3u) + \sh(5u)\big)\big]. 
\end{align*}
Using now the Mathematica \verb9Maximize9 and \verb9Minimize9 commands, one finds that 
$-203<F_{41}^{(6)}\le768$ and hence $|F_{41}^{(6)}|\le768$ on the interval $[0,u_*]$, which yields  
$|F_{41}(u)|\le\frac{768}{6!}\,u^6$ for all $u\in[0,u_*]$. 
Using \verb9Maximize9 and \verb9Minimize9 again (on the ratio $\frac{F_{42}(u)}{u^6}$, subject to the restriction $0<u\le u_*$), one finds that 
$|\frac{F_{42}(u)}{u^6}|\le\lim_{u\downarrow0}\frac{F_{42}(u)}{u^6}=6$ for all $u\in(0,u_*]$. 
So, for all $u\in[0,u_*]$ one has $|F_4(u)|\le|F_{41}(u)|+|F_{42}(u)|\le(\frac{768}{6!}+6)\,u^6$, whence 
\begin{equation*}
 	|f_4|\le\tfrac12\,(\tfrac{768}{6!}+6)<4
\end{equation*}
on the interval $(0,u_*]$. 
On the other hand, once again using \verb9Maximize9 and \verb9Minimize9, one sees that $|\frac{T_3}{g_4}|\le0.08$ and hence 
$|T_3 T_4|=|f_4||\frac{T_3}{g_4}|\le4\times 0.08=0.32$. 
Applying the command \verb9Maximize9 twice more yields $T_1<-0.47$ and $T_2-K<0$, so that 
indeed $\tLa-K=T_1 + ( T_3 T_4)+(T_2-K)<-0.47+0.32+0<0$ on $(0,u_*]$.  
Sublemma~\ref{subl:tLa} is now completely proved. 
\end{proof}

\section{Proof of Lemma~\ref{lem:BGH-h}}\label{BGH-h} 
This proof could be simplified using the mentioned result 
\eqref{eq:1,sqrt2}; however, we decided to present an independent proof here, which is not much more complicated. 
Let 
\begin{equation*}
	\De:=\De(x,u):=\tfrac{\sqrt{2\pi}}C\big[\tfrac12 h\big(U(x,u/x)\big)+\tfrac12 h\big(V(x,u/x)\big)-h(x)\big]. 
\end{equation*}
We have to show that $\De\le0$ for all pairs $(x,u)$ in the set 
\begin{equation*}
P:=\{(x,u)\in[\tfrac{15}{10},\infty)\times[u_*,\infty)\colon u<x\},	
\end{equation*}
the condition $u<x$ corresponding to the condition $a=a_n<1$. 
Introduce
\begin{align*}
\De_1:=\De_1(x,u)&:=\pd\De u\,
\exp\Big\{\frac{\left(u-x^2\right)^2}{2 \left(x^2-u^2\right)}\Big\}\,
\frac{\left(x^2-u^2\right)p_2(x,u)^2}
   {(u-1)x^2 \left(x^2-u\right)p_1(x,u)},\\
\De_2:=\De_2(x,u)&:=\pd{\De_1}u
\exp\Big\{\frac{2 u x^2}{x^2-u^2}\Big\}\\
&\times
\frac{(u-1)^2 (x-u)^2 (u+x)^2\left(x^2-u\right)^2 p_1(x,u)^2 p_3(x,u)^3 }
{p_2(x,u)},  
\end{align*}
where
\begin{equation}\label{eq:p's}
\left.
\begin{aligned}
	p_1(x,u)&:= x^2 (11 + x^2) - (10 u^2 + 2 u x^2),\\
p_2(x,u)&:=x^2 (9 + x^2) - (8 u^2 + 2 u x^2),\\
p_3(x,u)&:=x^2 (9 + x^2) - (8 u^2 - 2 u x^2)
\end{aligned}	
\right\}
\end{equation}
Consider also the set 
\begin{equation*}
\tP:=\{(x,u)\in[\tfrac{15}{10},\infty)\times[\tfrac4{10},\infty)\colon u<x\},	
\end{equation*}
which is slightly larger than $P$. 
Using e.g.\ the Mathematica command \verb9Reduce9, one can see that on the set $\tP$ the polynomials $p_1$, $p_2$, and $p_3$ are positive, and so, $\De_1(x,u)$ and $\De_2(x,u)$ are equal in sign to $(u-1)\pd\De u(x,u)$ and $\pd{\De_1}u(x,u)$, respectively, for all points $(x,u)\in P$ with $u\ne1$; note here that $u<x<x^2$ for all $(x,u)\in\tP$. 

Note also that $\De_2(x,u)$ is a polynomial in $(x,u)$ (of degree 24 in $x$, and 14 in $u$). Using again \verb9Reduce9, one finds that $\De_2>0$ on $\tP$. 
It follows that $\De_1(x,u)$ increases in $u\in[\frac4{10},1)$ and in $u\in(1,x]$ for each $x\in[\tfrac{15}{10},\infty)$. 
Moreover, $\De_1(x,x-)=-\frac12<0$, whence $\De_1(x,u)<0$ for all $(x,u)\in\tP$ such that $u>1$. 

On the other hand, one has 

\begin{subl}\label{subl:De1(x,u*)>0}
$\De_{1*}:=\De_{1*}(x):=\De_1(x,\frac4{10})>0$ for all $x\in[\tfrac{15}{10},\infty)$.
\end{subl}

We shall prove this a bit later in this section. 
Since $\De_1(x,u)$ increases in $u$, it follows that 
$\De_1(x,u)>0$ for all $(x,u)\in\tP$ such that $u<1$. 
Recalling that $\De_1(x,u)$ equals $(u-1)\pd\De u(x,u)$ in sign, we conclude that $\De(x,u)$ decreases in $u\in[\frac4{10},x]$ for each $x\in[\tfrac{15}{10},\infty)$. 
Therefore and because $u_*>\frac4{10}$, it remains to prove 


\begin{subl}\label{subl:De(x,u*)<0}
$\De_*(x):=\De(x,u_*)<0$ for all $x\in[\tfrac{15}{10},\infty)$ 
\end{subl}
-- as well as Sublemma~\ref{subl:De1(x,u*)>0}.

\begin{proof}[Proof of Sublemma~\ref{subl:De1(x,u*)>0}]
The derivative $\De_{1*}'(x)$ of $\De_{1*}(x)$ is of the form \break 
$R_1(x)e^{R_2(x)}$, where $R_1(x)$ is a certain rational expression in $x$. It follows (using again Mathematica, say), that $\De_{1*}'(x)\ge0$ iff $x\ge x_{1*}$, where $x_{1*}=3.62\dots$ is the only root in the interval $[\frac{15}{10},\infty)$ of a certain polynomial (of degree 20). So, the minimum of $\De_{1*}$ on $[\frac{15}{10},\infty)$ is attained at the point $x_{1*}$. 
Using the Mathematica command \verb9Reduce9 again, one finds that this minimum value is positive, which proves the sublemma.  
\end{proof}

\begin{proof}[Proof of Sublemma~\ref{subl:De(x,u*)<0}]
Introduce \big(cf.\ $\De_{1*}(x)$ in Sublemma~\ref{subl:De1(x,u*)>0}\big)
\begin{align}
\De_{*1}(x)&
:=\frac C{\sqrt{2\pi}}\,
\fd{}x \left(\frac{\De_*(x)}{h(x)}\right) \notag\\ 
&\times\frac{3814697265625}{51}\,
\exp\Big\{\frac{51 \left(301 x^2+51\right)}{31250 x^2-5202}\Big\}\,
\frac{\left(x^2-u_*^2\right)p_{3*}(x)^2}
   {xp_{4*}(x)},\label{eq:De*1}\\
\De_{*2}(x)&:=\fd{\De_{*1}(x)}x \notag\\
&\times\frac{95367431640625}{1224}\,
\exp\Big\{-\frac{12750 x^2}{15625 x^2-2601}\Big\}\,
\frac{\left(x^2-u_*^2\right)^2 p_{2*}(x)^3 p_{4*}(x)^2}
{xp_{3*}(x)},  \notag
\end{align}
where (recall \eqref{eq:p's})
\begin{gather*}
p_{2*}(x):=p_2(x,u_*),\quad p_{3*}(x):=p_3(x,u_*),\\ 
\begin{aligned}
p_{4*}(x):=-184559856669 + 1289843642871 x^2 &+ 244896587625 x^4 \\
&+ 85828328125 x^6.
\end{aligned}
\end{gather*}
Note that on the interval $[\tfrac{15}{10},\infty)$ the polynomials $p_{2*}$, $p_{3*}$, $p_{4*}$ are positive, and so, $\De_{*1}(x)$ and $\De_{*2}(x)$ are equal in sign to $\fd{}x \left(\frac{\De_*(x)}{h(x)}\right)$ and $\fd{\De_{*1}(x)}x$, respectively, for all $x\in[\tfrac{15}{10},\infty)$. 
Moreover, $\De_{*2}$ is a polynomial (of degree 20), which is negative on the interval $[\tfrac{15}{10},\infty)$, so that $\De_{*1}$ decreases on this interval. 
Next, $\De_{*1}(\tfrac{39}{10})=0.0042\ldots>0$ while $\De_{*1}(4)<0$. 
So, the maximum of $\De_*/h$ on the interval $[\tfrac{15}{10},\infty)$ is attained at some point between $\tfrac{39}{10}$ and $4$; it also follows that the maximum of $\De_{*1}$ on the interval $[\tfrac{39}{10},4]$ is less than $0.0043$. 
On the other hand, 
$\frac C{\sqrt{2\pi}}\,\fd{}x \left(\frac{\De_*(x)}{h(x)}\right)=\De_{*1}(x)M(x)$, where, in view of \eqref{eq:De*1}, $M(x)$ is the product of two positive expressions, one of which is 
$\exp\big\{-\frac{51 \left(301 x^2+51\right)}{31250 x^2-5202}\big\}$ 
and the other is a certain rational expression. 
Using the Mathematica command \verb9Maximize9 allows one to find the exact maximum of $M$ on $[\tfrac{39}{10},4]$, which is attained at $\tfrac{39}{10}$ and is less than $0.00552$; so, the maximum of $\frac C{\sqrt{2\pi}}\,\fd{}x \left(\frac{\De_*(x)}{h(x)}\right)$  
in $x\in[\tfrac{39}{10},4]$ is less than $0.0043\times0.00552<\frac1{40000}$.  
Thus, the maximum of $\frac C{\sqrt{2\pi}}\,\De_*/h$ on the interval $[\tfrac{15}{10},\infty)$ is no greater than 
$\frac C{\sqrt{2\pi}}\,
\De_*(\tfrac{39}{10})/h(\tfrac{39}{10})+\frac1{40000}\,(4-\tfrac{39}{10})=-0.000018\ldots<0$. 
This completes the proof of Sublemma~\ref{subl:De(x,u*)<0}. 
\end{proof}

\bibliographystyle{abbrv}


{\footnotesize
\bibliography{C:/Users/Iosif/Documents/mtu_home01-30-10/bib_files/citations}
}

\end{document}